\documentclass[12pt]{article}
\usepackage{amssymb,amsthm,amsmath,amsfonts,latexsym,tikz,hyperref,shuffle}
\usepackage[hmargin=.7in,vmargin=1in]{geometry}

\newtheorem{thm}{Theorem}[section]
\newtheorem{prop}[thm]{Proposition}
\newtheorem{cor}[thm]{Corollary}
\newtheorem{lem}[thm]{Lemma}
\newtheorem{conj}[thm]{Conjecture}
\newtheorem{exa}[thm]{Example}
\newtheorem{question}[thm]{Question}

\newcommand{\shu}{\shuffle}

\DeclareMathOperator{\cdes}{cdes}
\DeclareMathOperator{\exc}{exc}
\DeclareMathOperator{\Pk}{Pk}
\DeclareMathOperator{\pk}{pk}
\DeclareMathOperator{\cpk}{cpk}
\DeclareMathOperator{\Tr}{Tr}

\newcommand{\ben}{\begin{enumerate}}
\newcommand{\een}{\end{enumerate}}
\newcommand{\ble}{\begin{lem}}
\newcommand{\ele}{\end{lem}}
\newcommand{\bth}{\begin{thm}}
\renewcommand{\eth}{\end{thm}}
\newcommand{\bpr}{\begin{prop}}
\newcommand{\epr}{\end{prop}}
\newcommand{\bco}{\begin{cor}}
\newcommand{\eco}{\end{cor}}
\newcommand{\bcon}{\begin{conj}}
\newcommand{\econ}{\end{conj}}
\newcommand{\bde}{\begin{defn}}
\newcommand{\ede}{\end{defn}}
\newcommand{\bex}{\begin{exa}}
\newcommand{\eex}{\end{exa}}
\newcommand{\barr}{\begin{array}}
\newcommand{\earr}{\end{array}}
\newcommand{\btab}{\begin{tabular}}
\newcommand{\etab}{\end{tabular}}
\newcommand{\beq}{\begin{equation}}
\newcommand{\eeq}{\end{equation}}
\newcommand{\bea}{\begin{eqnarray*}}
\newcommand{\eea}{\end{eqnarray*}}
\newcommand{\bal}{\begin{align*}}
\newcommand{\bce}{\begin{center}}
\newcommand{\ece}{\end{center}}
\newcommand{\bpi}{\begin{picture}}
\newcommand{\epi}{\end{picture}}
\newcommand{\bpp}{\begin{picture}}
\newcommand{\epp}{\end{picture}}
\newcommand{\bfi}{\begin{figure} \begin{center}}
\newcommand{\efi}{\end{center} \end{figure}}
\newcommand{\bprf}{\begin{proof}}
\newcommand{\eprf}{\end{proof}\medskip}

\newcommand{\bsl}{\begin{slide}{}}
\newcommand{\esl}{\end{slide}}
\newcommand{\bfr}{\begin{frame}}
\newcommand{\efr}{\end{frame}}

\newcommand{\hqed}{\hfill \qed}

\newcommand{\eqed}[1]{$\textcolor{white}{\qed}\hfill{\dil#1}\hfill\qed$}

\newcommand{\ol}{\overline}

\newcommand{\hs}[1]{\hspace{#1}}
\newcommand{\hso}[1]{\hspace{-1pt}}
\newcommand{\vs}[1]{\vspace{#1}}
\newcommand{\qmq}[1]{\quad\mbox{#1}\quad}

\newcommand{\emp}{\emptyset}

\newcommand{\sbs}{\subset}
\newcommand{\sbe}{\subseteq}

\newcommand{\iso}{\cong}

\newcommand{\case}[4]{\left\{\barr{ll}#1&\mbox{#2}\\#3&\mbox{#4}\earr\right.}
\newcommand{\fl}[1]{\lfloor #1 \rfloor}

\def\<{\langle}
\def\>{\rangle}

\newcommand{\spn}[1]{\langle{#1}\rangle}

\newcommand{\ra}{\rightarrow}

\newcommand{\de}{\delta}
\newcommand{\ep}{\epsilon}

\newcommand{\io}{\iota}
\newcommand{\ka}{\kappa}

\newcommand{\si}{\sigma}
\renewcommand{\th}{\theta}

\newcommand{\bbN}{{\mathbb N}}

\newcommand{\fS}{{\mathfrak S}}

\DeclareMathOperator{\Av}{Av}

\DeclareMathOperator{\des}{des}
\DeclareMathOperator{\Des}{Des}

\DeclareMathOperator{\inv}{inv}

\DeclareMathOperator{\maj}{maj}

\DeclareMathOperator{\st}{st}

\newcommand{\dil}{\displaystyle}

\begin{document}
\pagestyle{plain}

\title{Cyclic Pattern Containment and Avoidance
}
\author{Rachel Domagalski\\[-5pt]
\small Department of Mathematics, Michigan State University,\\[-5pt]
\small East Lansing, MI 48824-1027, USA, {\tt domagal9@msu.edu}\\
Jinting Liang\\[-5pt]
\small Department of Mathematics, Michigan State University,\\[-5pt]
\small East Lansing, MI 48824-1027, USA, {\tt liangj26@msu.edu}\\
Quinn Minnich\\[-5pt]
\small Department of Mathematics, Michigan State University,\\[-5pt]
\small East Lansing, MI 48824-1027, USA, {\tt minnichq@msu.edu}\\
Bruce E. Sagan\\[-5pt]
\small Department of Mathematics, Michigan State University,\\[-5pt]
\small East Lansing, MI 48824-1027, USA, {\tt bsagan@msu.edu}\\
Jamie Schmidt\\[-5pt]
\small Department of Mathematics, Michigan State University,\\[-5pt]
\small East Lansing, MI 48824-1027, USA, {\tt schmi710@msu.edu}\\
Alexander Sietsema\\[-5pt]
\small Department of Mathematics, Michigan State University,\\[-5pt]
\small East Lansing, MI 48824-1027, USA, {\tt sietsem6@msu.edu}
}

\date{\today\\[10pt]
	\begin{flushleft}
	                                       \\[5pt]
	\end{flushleft}}

\maketitle

\begin{abstract}
 The study of pattern containment and avoidance for linear permutations is a well-estab\-lished area of enumerative combinatorics.  A cyclic permutation is the set of all rotations of a linear permutation.  Callan initiated the study of permutation avoidance in cyclic permutations and characterized the avoidance classes for all single permutations of length $4$.  We continue this work.  In particular, we establish a cyclic variant of the Erd\H{o}s-Szekeres Theorem that any linear permutation of length $mn+1$ must contain either the increasing pattern of length $m+1$ or the decreasing pattern of length $n+1$.  We then derive results about avoidance of multiple patterns of length $4$. We also determine generating functions for the cyclic descent statistic on these classes.  Finally, we end with various open questions and avenues for future research.
\end{abstract}

\section{Introduction}

We first review some notions from the well-studied theory of patterns in (linear) permutations.  More information on this topic can be found in the texts of B\'ona~\cite{bon:cp}, Sagan~\cite{sag:aoc}, or Stanley~\cite{sta:ec1,sta:ec2}.
Let $\bbN$  be the nonnegative  integers.  If $m,n\in\bbN$ then we define $[m,n]=\{m,m+1,\ldots,n\}$ which we abbreviate to $[n]=[1,n]$ when $m=1$.  Consider the symmetric group $\fS_n$ of all permutations $\pi=\pi_1\pi_2\ldots\pi_n$ of $[n]$ written in one-line notation.  
We call $n$ the {\em length} of $\pi$ and write $|\pi|=n$.  We will sometimes put commas between the elements of $\pi$ for readability.
We say that two sequences of distinct integers $\pi=\pi_1\ldots\pi_k$ and $\si=\si_1\ldots\si_k$ are {\em order isomorphic}, written $\pi\iso\si$, whenever $\pi_i<\pi_j$ if and only if $\si_i<\si_j$.   If $\si\in\fS_n$ and $\pi\in\fS_k$ then {\em $\si$ contains $\pi$ as a pattern} if there is a subsequence $\si'$ of $\si$ with $|\si'|=k$ and $\si'\iso\pi$.  If no such subsequence exists then {\em $\si$ avoids $\pi$}.  We use the notation
$$
\Av_n(\pi) =\{\si\in\fS_n \mid \text{$\si$ avoids $\pi$}\}
$$
for the {\em avoidance class} of $\pi$.
For example $\si=42351$ contains the pattern $\pi=3241$ because of the subsequence $4251$ among others.  But it avoids $1234$ because it has no increasing subsequence of length $4$.  One can extend this notion to sets of permutations $\Pi$ by letting
$$
\Av_n(\Pi) = \{\si\in\fS_n \mid \text{$\si$ avoids all $\pi\in\Pi$}\}
=\bigcap_{\pi\in\Pi} \Av_n(\pi).
$$

A famous theorem of Erd\H{o}s and Szekeres~\cite{ES:ctg} can be stated in terms of pattern containment and avoidance.  Let
$$
\io_n=12\ldots n
$$
and
$$
\de_n = n\ldots 21
$$
be the increasing and decreasing permutations of length $n$, respectively.
\begin{thm}[\cite{ES:ctg}]
\label{ESThm}
Suppose $m,n\in\bbN$.  Then any $\si\in\fS_{mn+1}$ contains either $\io_{m+1}$ or $\de_{n+1}$.  This is the best possible in that there exist permutations in $\fS_{mn}$ which avoid both $\io_{m+1}$ and $\de_{n+1}$.\hqed
\end{thm}

\begin{figure}
    \centering
    \begin{tikzpicture}
    \draw (1,1) grid (5,5);
    \fill(1,4) circle(.1);
    \fill(2,2) circle(.1);
    \fill(3,3) circle(.1);
    \fill(4,5) circle(.1);
    \fill(5,1) circle(.1);
    \draw(1,0.5) node{$1$};
    \draw(2,0.5) node{$2$};
    \draw(3,0.5) node{$3$};
    \draw(4,0.5) node{$4$};
    \draw(5,0.5) node{$5$};
    \draw(0.5,1) node{$1$};
    \draw(0.5,2) node{$2$};
    \draw(0.5,3) node{$3$};
    \draw(0.5,4) node{$4$};
    \draw(0.5,5) node{$5$};
    \end{tikzpicture}
\hs{60pt}
     \begin{tikzpicture}
     \draw[thick,dotted, ->] (0,1)--(0,3);
     \draw[thick,dotted, ->] (6,1)--(6,3);
     \draw[thick,dotted] (0,3)--(0,5) (6,3)--(6,5) (0,1)--(6,1) (0,5)--(6,5);
    \draw (1,1) grid (5,5);
    \fill(1,4) circle(.1);
    \fill(2,2) circle(.1);
    \fill(3,3) circle(.1);
    \fill(4,5) circle(.1);
    \fill(5,1) circle(.1);
    \draw(1,0.5) node{$1$};
    \draw(2,0.5) node{$2$};
    \draw(3,0.5) node{$3$};
    \draw(4,0.5) node{$4$};
    \draw(5,0.5) node{$5$};
    \draw(-0.5,1) node{$1$};
    \draw(-0.5,2) node{$2$};
    \draw(-0.5,3) node{$3$};
    \draw(-0.5,4) node{$4$};
    \draw(-0.5,5) node{$5$};
    \end{tikzpicture}
    \caption{The graph of $42351$ on the left and of $[42351]$ on the right}
    \label{42351}
\end{figure}
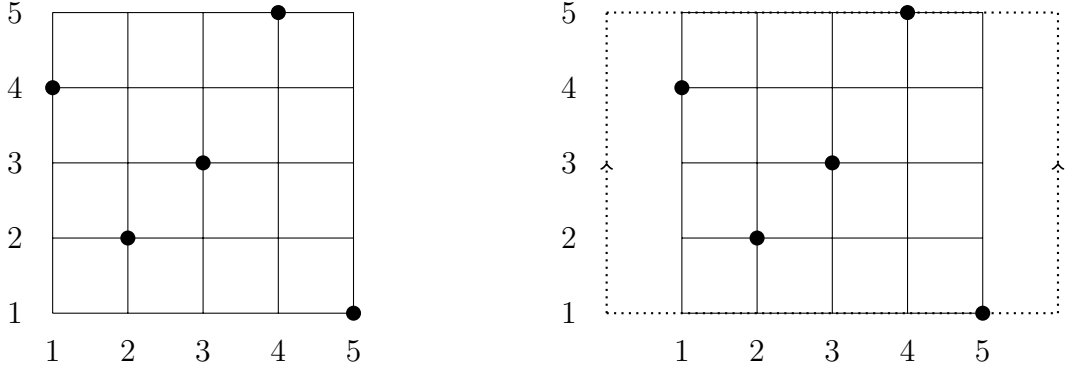

The {\em diagram} of  $\pi\in\fS_n$ is the collection of points $(i,\pi_i)$ in the first quadrant of the Cartesian plane.  
The graphical representation of $\pi=42351$ is given on the left in Figure~\ref{42351}.  It follows that we can act on $\pi$ with the dihedral group of the square 
$$
D_4 =\{\rho_0,\rho_{90},\rho_{180},\rho_{270},
r_0,r_1,r_{-1},r_\infty\}
$$
where $\rho_\th$ is rotation counterclockwise through $\th$ degrees and $r_m$ is reflection in a line of slope $m$.  We wish to write some of these rigid motions in terms of the one-line notation for $\pi=\pi_1\pi_2\ldots\pi_n$.  Reflection in a vertical line gives the {\em reversal} of $\pi$ which is
$$
\pi^r=\pi_n\ldots\pi_2\pi_1.
$$
Similarly, reflection in a horizontal line results in the {\em complement} of $\pi$
$$
\pi^c = n+1-\pi_1,\hs{7pt} n+1-\pi_2,\hs{7pt} \ldots,\hs{7pt} n+1-\pi_n.
$$
Combining these two operations gives rotation by $180$ degree or {\em reverse complement}
$$
\pi^{rc}= n+1-\pi_n,\hs{7pt}  \ldots,\hs{7pt}  n+1-\pi_2,\hs{7pt}  n+1-\pi_1.
$$
We apply any of these operations to sets of permutations by applying them to each element of the set.

We can use diagrams to inflate permutations.  
If we are given $\pi=\pi_1\pi_2\ldots\pi_n\in\fS_n$ and
permutations $\si_1,\si_2,\ldots,\si_n$ then the {\em inflation of $\pi$ by the $\si_i$} is the permutation
$\pi\spn{\si_1,\si_2,\ldots,\si_n}$ whose diagram is obtained from that of $\pi$ by replacing each vertex $(i,\pi_i)$ by a copy of $\si_i$.
For example, given
$\pi=132$  and $\si_1,\si_2,\si_3$ then a schematic of the diagram of $132\spn{\si_1,\si_2,\si_3}$ is given on the right in Figure~\ref{132}.  More concretely, if $\si_1=21$, $\si_2=1$, and $\si_3=213$ then
$$
132\spn{\si_1,\si_2,\si_3}=216435.
$$

We say that patterns $\pi$ and $\pi'$ are {\em Wilf equivalent}, written $\pi\equiv\pi'$, if
$\#\Av_n(\pi)=\#\Av_n(\pi')$ for all $n\in\bbN$ where the hash symbol denotes cardinality.  This definition extends in the obvious way to sets of patterns.
Note that if $\pi$ and $\pi'$ are Wilf equivalent then both must be in the same $\fS_n$.  It is easy to see that if $\phi\in D_n$ then 
$\pi\equiv\phi(\pi)$ and so these are called {\em trivial Wilf equivalences}.  As is well known, all elements of $\fS_3$ are Wilf equivalent.
\bth 
If $\pi\in\fS_3$ then
$$
\#\Av_n(\pi) = C_n
$$
where $C_n=\frac{1}{n+1}\binom{2n}{n}$ it the $n$th Catalan number.\hqed
\eth
\noindent Trivial Wilf equivalence carries over to sets $\Pi$ of permutations.  Simion and Schmidt~\cite{SS:rp} determined all Wilf equivalences among the $\Av_n(\Pi)$ for all $\Pi\sbe\fS_3$.

A {\em permutation statistic} is a map $\st:\uplus_{n\ge0} \fS_n\ra S$ where $S$ is some set.  Many statistics are based on the {\em descent set} statistic which is
$$
\Des\pi=\{i \mid \pi_i>\pi_{i+1}\}.
$$
The elements $i\in\Des\pi$ are called {\em descents} and if $\pi_i<\pi_{i+1}$ then $i$ is called an {\em ascent}.
Four famous statistics related to $\Des$ are the
the {\em descent number} statistic
$$
\des\pi=\#\Des\pi
$$
the {\em major index} statistic
$$
\maj\pi=\sum_{i\in\Des\pi} i,
$$
the {\em inversion} statistic
$$
\inv\pi=\#\{(i,j) \mid \text{$i<j$ and $\pi_i>\pi_j$}\},
$$
and the {\em excedance statistic}
$$
\exc\pi=\#\{i \mid \pi(i)>i\}.
$$
Let $\st$ be a statistic whose range is $\bbN$ and $q$ be a variable. If $\Pi$ is a set of patterns then its avoidance class has a corresponding generating function
$$
F^{\st}_n(\Pi)=F^{\st}_n(\Pi;q)=\sum_{\si\in\Av_n(\Pi)} q^{\st\si}.
$$
Say that $\Pi$ and $\Pi'$ are {\em $\st$-Wilf equivalent} and write
$\Pi\stackrel{\st}{\equiv}\Pi'$ if 
$F^{\st}_n(\Pi)=F^{\st}_n(\Pi')$ for all $n\ge0$.  Clearly $\st$-Wilf equivalence implies Wilf equivalence.  The $\maj$- and $\inv$-Wilf equivalence classes for $\Pi\sbe\fS_3$ were determined by 
Dokos, Dwyer, Johnson, Sagan,  and Selsor~\cite{DDJSS:pps}.

\bfi
\begin{tikzpicture}
\draw (0,0) grid (2,2);
\fill (0,0) circle(.1);
\fill (1,2) circle(.1);
\fill (2,1) circle(.1);
\end{tikzpicture}
\hspace{50pt}
\begin{tikzpicture}
\draw (0,0) -- (2,0) -- (2,2) -- (0,2) -- (0,0);
\draw (2/3,0) -- (2/3,2/3) -- (0,2/3);
\draw (2,2/3) -- (4/3, 2/3) -- (4/3,2);;
\draw (2,4/3) -- (2/3,4/3) -- (2/3,2);
\node at (1/3,1/3){$\si_1$};
\node at (1,5/3){$\si_2$};
\node at (5/3,1){$\si_3$};
\end{tikzpicture}
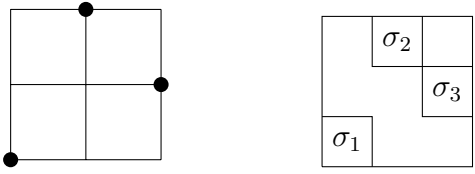
\caption{The diagram of $132$ (left) and $132\spn{\si_1,\si_2,\si_3}$ (right)}
\label{132}
\efi

If $\pi=\pi_1\pi_2\ldots\pi_n\in\fS_n$ then the corresponding {\em cyclic permutation} is the set of all rotations of $\pi$, denoted
$$
[\pi]=\{\pi_1\pi_2\ldots\pi_n,\hs{7pt} \pi_2\ldots\pi_n\pi_1,\hs{7pt} 
\ldots,\hs{7pt} \pi_n\pi_1\ldots,\pi_{n-1}\}.
$$
Continuing our example from the beginning of the section,
$$
[42351] = \{42351,\hs{7pt} 23514,\hs{7pt} 35142,\hs{7pt} 51423,\hs{7pt} 14235\}.
$$
If necessary, we will call permutations from $\fS_n$ {\em linear} to distinguish them from their cyclic cousins.  We also use square brackets to denote cyclic analogues of objects defined in the linear case.  For example, $[\fS_n]$ is the set of all cyclic permutations of length $n$.  We say a cyclic permutation $[\si]$ {\em contains $[\pi]$ as a pattern} if there is some rotation $\si'$ of $\si$ which contains $\pi$ linearly.  Otherwise $[\si]$ {\em avoids $[\pi]$}.
In our perennial example, even though $42351$ avoids $1234$ we have that $[42351]$ contains $[1234]$ since the rotation $14235$ has the copy $1235$ of this pattern.  Given a set $[\Pi]$ of cyclic patterns the cyclic avoidance class $\Av_n[\Pi]$ is defined as expected.
Note that when using a specific set of cyclic permutations the square brackets will be put around the permutations themselves, for example,
$\Av_n([\pi],[\pi'])$.
Callan~\cite{cal:pac} determined $\#\Av_n[\pi]$ for all $[\pi]\in[\fS_4]$.  Gray, Lanning, and Wang continued work in this direction considering cyclic packing of patterns~\cite{GLW:pcc1} and patterns in colored cyclic permutations~\cite{GLW:pcc2}.

The graph of a cyclic permutation $[\pi]$ is obtained by embedding the graph of $\pi$ on a cylinder.   This is indicated on the right in Figure~\ref{42351} by identifying the two dotted arrows.  Cyclic Wilf equivalence has the obvious definition.  But note that now there are fewer trivial cyclic Wilf equivalences since we need the chosen group element to preserved the cylinder, not just the square.  So the only trivial equivalences are
\begin{equation}
\label{cW=}
[\pi]\equiv [\pi^r]\equiv[\pi^c]\equiv[\pi^{rc}].    
\end{equation}

Certain linear permutation statistics have obvious cyclic analogues.  For example, 
if $\pi\in\fS_n$ then its {\em cyclic descent number} is
$$
\cdes[\pi] =\#\{i \mid \text{$\pi_i>\pi_{i+1}$ where subscripts are taken modulo $n$}\}.
$$
Note that this is well defined because the cardinality does not depend on which representative of $[\pi]$ is chosen.  To illustrate, $\pi=23514$ has cyclic descents at indices $3$ and $5$ so $\cdes[\pi]=2$.
The corresponding generating function
$F^{\cdes}_n[\Pi]$ where $[\Pi]$ is a set of cyclic permutations, and $\cdes$-Wilf equivalence should now need no definition.  
Note that $\cdes$ is another form of the excedance statistic on linear permutations.  In particular, if $\pi=\pi_1\pi_2\ldots\pi_n$ then
$$
\cdes[\pi]=\exc (\pi_n,\ldots,\pi_2,\pi_1)
$$
where $(\pi_n,\pi_{n-1}\ldots,\pi_1)$ is cycle notation for the linear permutation which, as a function, sends $\pi_i$ to $\pi_{i-1}$ for all $i$ modulo $n$.

The rest of this paper is organized as follows.  In the next section we prove a cyclic variant of the Erd\H{o}s-Szekeres Theorem.  This will be used in the sequel to show that, for certain $[\Pi]$, the set $\Av_n[\Pi]$ becomes empty for large enough $n$.  Section~\ref{pad}  will extend Callan's work by enumerating $\Av_n[\Pi]$ for $[\Pi]\sbs[\fS_4]$ consisting of two patterns.   One of our principle proof techniques will be the use of generating trees.  The following section will consider $[\Pi]$ with three or more patterns.  In Section~\ref{cdg} we will compute the cyclic descent generating functions for $[\Pi]\sbs[\fS_4]$, thus refining the previous enumerations.  We will end with a section of open problems and additional comments.

\section{A cyclic Erd\H{o}s-Szekeres Theorem}
\label{cEST}

In this section we will use the linear Erd\H{o}s-Szekeres Theorem to prove a cyclic analogue.  We will need a variant of the decreasing permutation $\de_n$ defined as follows.  Given nonnegative integers $n$ (the length), $d$ (the difference), and $s$ (the smallest value)  define the decreasing sequence
$$
\de_{n,d,s} = s+(n-1)d,\hs{7pt} s+(n-2)d,\hs{7pt} \ldots,\hs{7pt}
s+d,\hs{7pt} s.
$$
For example
$$
\de_{5,2,3} = 11, 9, 7, 5, 3.
$$

\begin{figure}
    \centering
     \begin{tikzpicture}[scale=.6]
     \draw[thick,dotted, ->] (0,1)--(0,8.5);
     \draw[thick,dotted, ->] (17,1)--(17,8.5);
     \draw[thick,dotted] (0,8.5)--(0,16) (17,8.5)--(17,16) (0,1)--(17,1) (0,16)--(17,16);
    \draw (1,1) grid (16,16);
    \fill(1,1) circle(.1);
    \fill(2,12) circle(.1);
    \fill(3,7) circle(.1);
    \fill(4,2) circle(.1);
    \fill(5,13) circle(.1);
    \fill(6,8) circle(.1);
    \fill(7,3) circle(.1); 
    \fill(8,14) circle(.1);
    \fill(9,9) circle(.1);
    \fill(10,4) circle(.1);    
    \fill(11,15) circle(.1);
    \fill(12,10) circle(.1);
    \fill(13,5) circle(.1);
    \fill(14,16) circle(.1);
    \fill(15,11) circle(.1);
    \fill(16,6) circle(.1);
    \foreach \x in {1,2,...,16}
          {\draw(\x,0.5) node{$\x$};
            \draw(-.5,\x) node{$\x$};
          }
    \end{tikzpicture}
    \caption{The graph of $[\si]$ when $m=5$ and $n=3$}
    \label{ESFig}
\end{figure}
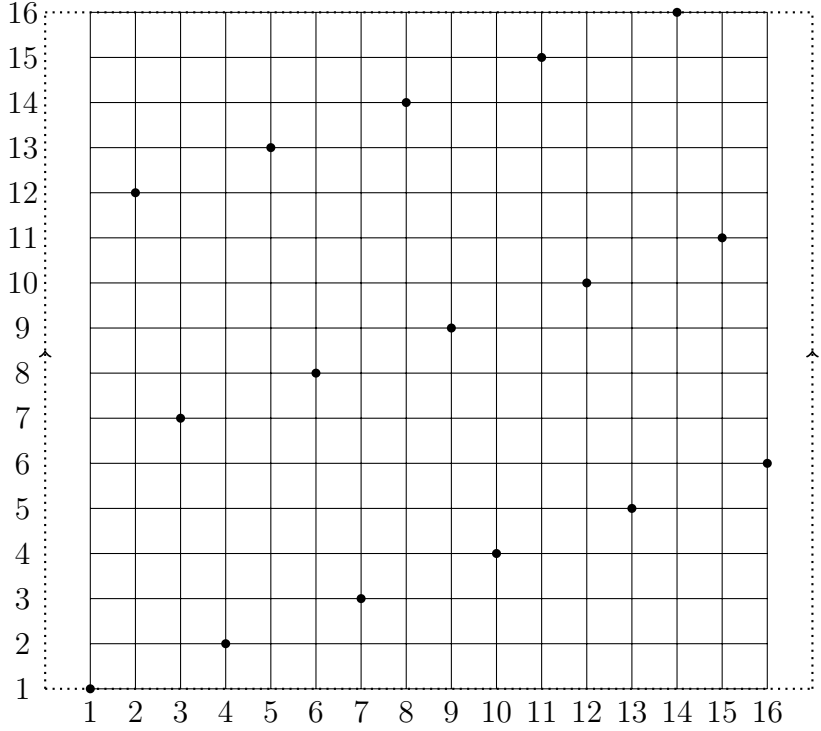

\begin{thm}
\label{EScyc}
Suppose $m,n\in\bbN$.  Then any $[\si]\in[\fS_{mn+2}]$ contains either $[\io_{m+2}]$ or $[\de_{n+2}]$.  This is the best possible in that there exist permutations in $[\fS_{mn+1}]$ which avoid both $[\io_{m+2}]$ and $[\de_{n+2}]$.
\end{thm}
\begin{proof}
To prove the first statement we can assume, by rotating $\pi$ if necessary, that 
$$
\si=\si_1,\si_2,\ldots,\si_{mn+1},mn+2.
$$
So $\si'=\si_1\si_2\ldots\si_{mn+1}\in\fS_{mn+1}$ and, by Theorem~\ref{ESThm}, contains a copy $\ka$ of either $\io_{m+1}$ or $\de_{n+1}$.
In the first case, the concatenation $\ka,mn+1$ is a copy of $[\io_{m+2}]$ in $[\pi]$.  In the second case, we have that 
$mn+1,\ka$ is a copy of $[\de_{n+2}]$ in $[\si]$.

To prove the second statement, consider the concatenation
$$
\si = 1,\hs{3pt} \de_{n,m,2},\hs{3pt} \de_{n,m,3},\hs{3pt} 
\ldots,\hs{3pt} \de_{n,m,m+1}.
$$
For example, when $m=5$ and $n=3$ then
$$
[\si] =[1, 12, 7, 2, 13, 8, 3, 14, 9, 4, 15, 10, 5, 16, 11, 6] 
$$
whose graph is shown in Figure~\ref{ESFig}.  Define $\si'$ by 
$\si=1\si'$ and note that $\si'$ can be written either as a disjoint union of $m$ decreasing subsequences of length $n$, or of $n$ increasing subsequences of length $m$.  In a linear permutation, any increasing subsequence can intersect any decreasing subsequence at most once.
So any increasing subsequence of $\si'$ has length at most $m$, and any decreasing subsequence has length at most $n$.  Now let $[\pi]$ be a subsequence of $[\si]$.  We consider two cases.

Suppose first that $[\pi]$ contains $1$.  If $[\pi]$ is increasing then rotate, if necessary, until $\pi=1\pi'$ for some $\pi'$ which is a subsequence of $\si'$.   But from the previous paragraph, $|\pi'|\le m$ which implies $|\pi|\le m+1$ as desired.  If $[\pi]$ is decreasing then we pick a representative $\pi=\pi'1$ and proceed as in the increasing case to get $|\pi|\le n+1$.

Now consider the possibility that $[\pi]$ does not contain $1$.  Again, we start with the subcase when $[\pi]$ is increasing.  Suppose, for simplicity, that $\pi$ contains an element of $x\in\de_{n,m,2}$ as the proof will be similar for the other deltas.  As before, $\pi$ can contain at most one element of each of $\de_{n,m,3}$ through $\de_{n,m,m+1}$.  Now $[\pi]$ can wrap around and pick up other elements.  But those elements must come before $x$.  And since $\de_{n,m,2}$ is decreasing, at most one other element can be added in this way.  It follows that $|\pi|\le m+1$.  On the other hand, if $[\pi]$ is decreasing then the proof is similar.  The only difference is that if one attempts to pick up elements of $\de_{n,m,2}$ before $x$ then this is impossible since such elements are larger than $x$ and $[\pi]$ is decreasing.  So $|\pi|\le n$ which is an even tighter bound.  This completes the demonstration of the theorem.
\end{proof}

\section{Pattern avoidance of doubletons}
\label{pad}

In this section we will enumerate $\Av_n[\Pi]$ for all $[\Pi]\sbs[\fS_4]$ with $\#[\Pi]=2$.  Any cyclic Wilf equivalences stated without proof are trivial.

Let us first dispose of the simplest singleton avoidance classes where $[\pi]\in[\fS_k]$ for $k<4$.  In $[\fS_2]$ there is only one cyclic permutation $[12]$ and it is easy to see that every $[\si]$ of length at least $2$ contains it.  In $[\fS_3]$ there are only the patterns $[123]$ and $[321]$, and these are only avoided by $[\de_n]$ and $[\io_n]$, respectively.

Callan~\cite{cal:pac} enumerated $\Av_n[\pi]$ for any given $[\pi]\in[\fS_4]$.  Recall the version of the Fibonacci numbers defined by $F_1=F_2=1$ and $F_n=F_{n-1}+F_{n-2}$ for $n\ge3$.   Unlike the case of linear permutations in $\fS_3$, there are no nontrivial Wilf equivalences.
\bth[\cite{cal:pac}]
For $n\ge2$ we have
\begin{align*}
\#Av_n[1234]=\#Av_n[1432]&= 2^n + 1 - 2n - \binom{n}{3},\\   \#Av_n[1243]=\#Av_n[1342]&= 2^{n-1}-n+1,\\[5pt]
\rule{130pt}{0pt}\#Av_n[1324]=\#Av_n[1423]&= F_{2n-3}.\rule{130pt}{0pt}\qed
\end{align*}
\eth

In presenting the enumerations for doubletons, we make the following conventions to facilitate locating a given result.  All cyclic patterns will be listed starting with $1$.  And all sets of cyclic patterns will be given in lexicographic order.  We will also use terms like ``just before" or ``just after" in $[\si]$ to refer the left-to-right order on the cylinder of a cyclic permutation in the form of Figure~\ref{42351}.  For example, in $[\si]=[42351]$ the $5$ comes just before $1$ and the $4$ just after.  We also say that an element $x$ is between $y$ and $z$ if it is in the subsequence of $[\si]$ traversed going left-to-right around the cylinder from $y$ to $z$.  Continuing our example, between $2$ and $5$ we have $3$, while between $5$ and $2$ we have $1$ and $4$.

One of our tools will be generating trees.  To the best of our knowledge, these trees were introduced by Chung, Grahamm, Hoggatt, and Kleiman~\cite{CGHK:nbp} for studying Baxter permutations.  Since then, they have become an integral technique in the theory of pattern avoidance~\cite{BBDFGG:gfg,bou:fcp,kre:pfs,Wes:gtc,wes:gtf}.  The {\em generating tree} for an avoidance class $\Av[\Pi]$, denoted $T[\Pi]$, has as its root the permutation $[12]$.  The children of any $[\si]\in\Av_n[\Pi]$ are all the $[\si']\in\Av_{n+1}[\Pi]$ which can be formed by inserting $n+1$ into one of the spaces of $[\si]$.  
A space, also called a {\em site}, where insertion of $n+1$ produces a permutation of the avoidance class is called {\em active} while the other spaces are {\em inactive}. 
A useful observation is that if a space is inactive it must be because inserting $n+1$ there results in copy of a forbidden pattern $[\pi]$ where $n+1$ plays the role of the largest element of $\pi$.
Once we have picked a representative $\si=\si_1\si_2\ldots\si_n$ for $[\si]$ we will label the spaces as $1,2,\ldots,n$ left to right where space $i$ comes between $\si_i$ and $\si_{i+1}$.
The nodes for $\Av_n[\Pi]$ will be said to be at {\em level $n$} in $T[\Pi]$.
We call the number of children of a vertex its {\em degree} which is denoted $\deg[\si]$.  Given $d\in\bbN$, suppose that every cyclic permutation with $\deg[\si]=d$ has children of degrees $c_1, c_2,\ldots, c_d$.  Then this is denoted by the {\em production rule}
$$
(d)\ra(c_1)(c_2)\ldots(c_d).
$$
There may be other nodes having some special characteristic $X$ which always produces nodes having characteristics $Y_1,Y_2,\ldots,Y_d$ which correspond to a production rule
$$
(X)\ra(Y_1)(Y_2)\ldots(Y_d).
$$
In particular, the characteristic of being the root of the tree is denote in a production rule by $(*)$.  We can also have production rules which mix numbers for degrees and letters for characteristics.  If $T[\pi]$ can be characterized by production rules, these can often be used to calculate $\#\Av_n[\Pi]$.

\begin{thm}
\label{1234,1243}
We have
$$
\{[1234],[1243]\} \equiv \{[1234],[1342]\} \equiv 
\{[1243],[1432]\} \equiv \{[1342],[1432]\}. 
$$
And for $n\ge3$
$$
\#\Av_n([1234],[1342])=2(n-2).
$$
\end{thm}
\begin{proof}
We claim that $T=T([1234],[1342])$ has the following production rules
\begin{align*}
(*)&\ra (2)(2),\\
(1)&\ra(1),\\
(2)&\ra(1)(2).
\end{align*}
Once these are proved then the enumeration follows easily since one can inductively show that, for $n\ge3$, level $n$ consists of two nodes of degree $2$ and $2(n-3)$ nodes of degree $1$.  

It is easy to check the production rule at  levels $n=2$ and $3$, so we assume that $n\ge4$ and also that $[\si]\in\Av_n([1234],[1342])$. 
First of all, note that the site before $n$  is always active.  For if it were not then the result $[\si']$ of inserting $n+1$ would have a copy $\ka$ of one of the patterns containing $n+1$.  But $n$ can not be in $\ka$ since neither of the patterns have $4$ followed  immediately in the cycle by $3$.  So replacing $n+1$ by $n$ in $\ka$ would give a forbidden pattern in $[\si]$ which is a contradiction.  Thus every $[\si]$ at  has at least one child.  Also $\si$ has at most two children.  For suppose 
$$
\si'=n+1,\rho,n,\tau
$$
is the result of inserting $n+1$ in $\si$.  It follows that $|\rho|\le 1$ since if $\rho\ge2$ then $[\si']$ has a copy of either $[4123]$ or $[4213]$.  Thus $n+1$ must be inserted either directly before $n$ or two elements before $n$.

Now consider 
\begin{equation}
  \label{deep}
\de= n, n-1, \ldots, 3,2,1,\qmq{and}
\ep=n, n-1,\ldots, 3,1,2.  
\end{equation}
It is easy to check that both sites $n$ and $n-1$ are active in these permutations and so both have degree $2$.  It is also obvious that if one inserts $n+1$ in site $n$ in either permutation then one gets another permutation of the same form.

From what we have done, we can finish the proof if we show that $\deg[\si]=2$ implies $[\si]=[\de]$ or $[\si]=[\ep]$.
Write 
$$
\si = n\rho m
$$
where $m$ is the last element of $\si$ and $\rho$ is everything between $n$ and $m$.  Since $\deg[\si]=2$, site $n-1$ is active and inserting $n+1$ there yields
$$
\si'=n,\rho,n+1,m.
$$
Then $m\le 2$ since otherwise $[\si]$ contains a copy of $[4123]$ or $[4213]$ since $n\ge4$.  In the case $m=1$ we must have $\rho$ decreasing.  For if there is an ascent $x<y$ in $\rho$ then $[\si']$ contains $[x,y,n+1,1]$ which is a copy of $[2341]$, a contradiction.  So in this case $\rho$ is decreasing and $\si=\de$.
The other possibility is that $m=2$.  This forces the last element of $\rho$ to be $1$.  For if $1$ is elsewhere and $x$ is the last element of $\rho$ then then $[\si']$ contains
$[1,x,n+1,2]$ which is contradictory copy of $[1342]$.  Similarly to the first case, one can now show that $\rho$ is decreasing and so $\si=\ep$ as desired.
\end{proof}

Comparing our next result with the previous one will provide our first nontrivial Wilf equivalence.
\begin{thm}
\label{1234,1324}
We have
$$
\{[1234],[1324]\} \equiv \{[1423],[1432]\}. 
$$
And for $n\ge3$
$$
\#\Av_n([1234],[1324])=2(n-2).
$$
\end{thm}
\begin{proof}
Let $D$ stand for the decreasing permutation and $E$ for the decreasing permutation with its largest two elements swapped. We consider the root $[12]$ to be of type $D$.
We will show that $T=T([1234],[1324])$ has production rules
\begin{align*}
(1)&\ra (1),\\
(D)&\ra (D)(E),\\
(E)&\ra (1)(1).
\end{align*}
It follows by induction that level $n\ge3$ of $T$ has a $D$, an $E$, and $2(n-3)$ nodes of degree one, proving the theorem.

The same demonstration as in the previous theorem shows that the site before $n$ in any $[\si]\in\Av_n([1234],[1324])$ is active.  So again, every such permutation has at least one child.
 Also, every $[\si]$ has at most two children.  Indeed,  write 
\begin{equation}
\label{si=1}
\si=1\si_2\ldots\si_n 
\end{equation}
and put $n+1$ in site $i\ge 3$.  Then $1,\si_2,\si_3,n+1$ is a copy of either $1234$ or $1324$, another contradiction.

Now consider  permutations corresponding to $D$ and $E$ at level $n$
\begin{equation}
   \label{deep2} 
\de =1,n,n-1,n-2,n-3,\ldots,2 \qmq{and}
\ep =1,n-1,n,n-2,n-3,\ldots,2.  
\end{equation}
It is easy to check that both sites $1$ and $2$ are active in $\de$, $\ep$.  So, by the previous paragraph, they both have degree $2$.  Furthermore, the two children of $\de$ have the form $D$ and $E$.

We will be done if we can show that $[\si]$ having two children implies $[\si]=[\de]$ or $[\ep]$.  Write $\si$ as in~\eqref{si=1}.
Since the active sites must be $1$ and $2$, and the site before $n$ must be active, either $\si_2=n$ or $\si_3=n$.
If $\si_2=n$ and there is an ascent $x<y$ in the rest of the permutation, then after inserting $n+1$ in position $2$ we have $[x,y,n,n+1]$ which is a copy of $[1234]$, a contradiction.  So in this case $[\si]=[\de]$.
Alternatively, suppose $\si_3=n$.  This forces $\si_2=n-1$, since if $\si_2=x<n-1$ then $n-1$ comes after $n$.  But inserting $n+1$ in position $1$ gives
$[x,n,n-1,n+1]$ which is a copy of $[1324]$. And similarly to the first case we see that the rest of $\si$ is decreasing.  The result is that $[\si]=[\ep]$.  This completes the proof.
\end{proof}

\begin{thm}
\label{1234,1423}
We have
$$
\{[1234],[1423]\} \equiv \{[1324],[1432]\}. 
$$
And for $n\ge1$
$$
\#\Av_n([1234],[1423])=1+\binom{n-1}{2}.
$$
\end{thm}
\begin{proof}
Suppose $[\si]\in\Av_n([1234],[1423])$ and write
\begin{equation}
\label{1rt}
 \si=1\rho n\tau   
\end{equation}
where $\rho$ and $\tau$ are the subsequences between $1$ and $n$, and between $n$ and $1$, respectively.  Now $\rho$ and $\tau$ must be decreasing since $[\si]$ avoids $[1234]$ and $[1423]$, respectively.  Furthermore, $\rho$ must consist of consecutive integers since, if not, then we have
$x<y<z$ such that $1zxny$ is a subsequence of $\si$.  So $[xnyz]$ is a copy of $[1423]$ in $[\si]$, which is a contradiction.  Conversely, it is easy to check that if $\si$ has the form~\eqref{1rt} with $\rho$ and $\tau$ decreasing and $\rho$ consecutive then $[\si]\in\Av_n([1234],[1423])$.  So we have characterized the elements of this class.

To finish the enumeration, if $\rho=\emp$ there is one corresponding $\si$.  But if $\rho\neq\emp$ then choosing the smallest and largest element of $\rho$ from the elements $2,3,\ldots,n-1$ completely determines $\si$.  Since these two elements could be equal, we are choosing $2$ elements from $n-2$ elements with repetition which is counted by $\binom{n-1}{2}$.
\end{proof}

The following result follows immediately from Theorem~\ref{EScyc}
\begin{thm}
\label{1234,1432}
We have
$$
\#\Av_n([1234],[1432])=0
$$
for $n\ge6$.\hqed
\end{thm}

We now have, by comparison with Theorem~\ref{1234,1423}, another nontrivial Wilf equivalence.  
\begin{thm}
\label{1243,1324}
We have
$$
\{[1243],[1324]\} \equiv \{[1243],[1423]\}
\equiv \{[1324],[1342]\} \equiv \{[1342],[1423]\}. 
$$
And for $n\ge1$
$$
\#\Av_n([1324],[1342])=1+\binom{n-1}{2}.
$$
\end{thm}
\begin{proof}
Take $[\si]\in\Av_n([1324],[1342])$ and write $\si$ as in~\eqref{1rt}.  Then $\rho$ is increasing since $[\si]$ avoids $[1324]$.  And every element of $\rho$ is smaller than every element of $\tau$ since $[\si]$ avoids $[1342]$.
To avoid a copy of one of the forbidden patterns containing the $1$ of $\si$ we must have that $\tau$ avoids $213$ and $231$.  And to avoid a copy of $[1324]$ where $n$ plays the role of $4$, it must be that $\tau$ avoids $132$.  The $\tau$ which avoid these three pattern are exactly those which are inflations of the form
$\tau=21\spn{\de_k,\io_l}$ for some $k,l\ge0$
(see the chart on page 2773 of~\cite{DDJSS:pps}).
Absorbing the $1$ and $n$ of $\si$ into $\rho$ and $\tau$, respectively, we see that 
\begin{equation}
\label{132idi}
    \si=132\spn{\io_j,\de_k,\io_l}
\end{equation}
where $j,k\ge1$ and $l\ge0$.  
Again, it is not hard to check that for every $\si$ of this form we have $[\si]\in\Av_n([1324],[1342])$.

To enumerate these $\si$, we distinguish two cases.  If $l\ge2$ then picking the smallest and largest elements of the copy of $\io_l$ from $2,3,\ldots,n-1$ completely determines $\si$ .  So in this case there are $\binom{n-2}{2}$ choices.  If $l\le1$ then the copy of $\io_l$ can be appended to the copy of $\de_k$ so that $\si=12[\io_j,\de_{n-j}]$.  Since we must have $1$ and $n$ in the ascending and decreasing subsequences, there are now $n-1$ choices.  Adding the two counts given the desired result.
\end{proof}

\begin{thm}
\label{1243,1342}
For $n\ge4$ we have
$$
\#\Av_n([1243],[1342])=4.
$$
\end{thm}
\begin{proof}
Take $[\si]\in\Av_n([1243],[1342])$ and write $\si$ as in~\eqref{1rt}.  Then $\rho$ and $\tau$ can not both be nonempty.  For if $x\in\rho$ and $y\in\tau$ then 
$1xny$ is a copy of either $1243$ or $1342$.

Assume first that $\rho=\emp$ so that 
\begin{equation}
\label{1nt}
\si=1n\tau.    
\end{equation}
Then $\tau$ must be increasing or decreasing.  For suppose it was neither.  Then it would contain a copy of one of the patterns $132$, $231$, $213$, or $312$.  In the first two cases this would give, together with the $1$,  a copy of $1243$ or $1342$ in $\si$.  And in the last two cases, prepending $n$ gives a copy of $4213$ or $4312$.  Conversely, if $\si$ is given by~\eqref{1nt} with $\tau$ increasing or decreasing then it is easy to verify that  $[\si]\in\Av_n([1243],[1342])$.

Using the same ideas, one can also show that if $\tau=\emp$ then one gets exactly two elements of $\Av_n([1243],[1342])$, of the form $\si=1\rho n$ where $\rho$ is either increasing or decreasing.  Thus there are a total of four elements in the avoidance class.
\end{proof}

\begin{thm}
\label{1324,1423}
For $n\ge3$ we have
$$
\#\Av_n([1324],[1423])=2^{n-2}.
$$
\end{thm}
\begin{proof}
Take $[\si] \in \Av_n([1324],[1423])$ and write
$$
\si = n,\rho,n-1,\tau.
$$
Similar to the previous proof, one of $\rho$ or $\tau$ must be empty since otherwise $4132$ or $4231$ is a pattern in $\si$.  If $\rho=\emp$ then one shows similarly that $n-2$ either begins or ends $\tau$.  Continuing in this manner, we see that there are $2$ choices for the positions of $n-1,n-2,\ldots,2$.  Checking, as usual, that all such permutations are actually in the avoidance set,  the enumeration follows.
\end{proof}

\section{Three or more patterns}
\label{tmp}

We will now compute $\#\Av_n[\Pi]$ for $\Pi\sbe\fS_n$ having $\#\Pi\ge3$.  Will will not consider those $[\Pi]$ containing both $[1234]$ and $[1432]$ since for such classes $\#\Av_n[\Pi]=0$ for $n\ge6$ as in Theorem~\ref{1234,1432}.

\begin{thm}
\label{1234,1243,1324}
We have
$$
\{[1234],[1243],[1324]\} \equiv \{[1234],[1324],[1342]\}
\equiv \{[1243],[1423],[1432]\} \equiv \{[1342],[1423],[1432]\}. 
$$
And for $n\ge4$
$$
\#\Av_n([1234],[1324],[1342])=3.
$$
\end{thm}
\begin{proof}
If $[\si]\in\Av_n([1234],[1324],[1342])$ then $[\si]$ avoids $[1324]$ and $[1342]$.  So, by the proof of Theorem~\ref{1243,1324},  we can write $\si$ in the form~\eqref{132idi} for $j,k,l\ge1$.  But since $[\si]$ also avoids $[1234]$ we must have $j+l\le 3$.  For the same reason, $j\le 2$ since if $j=3$ then the copy of $\io_3$ and one element of the copy of $\de_k$ would form a $[1234]$.  Thus the only possibilities are $(j,l)=(1,1)$, $(1,2)$, or $(2,1)$ which proves the result.
\end{proof}

\begin{thm}
\label{1234,1243,1342}
We have
$$
\{[1234],[1243],[1342]\} \equiv \{[1243],[1342],[1432]\}. 
$$
And for $n\ge5$
$$
\#\Av_n([1234],[1243],[1342])=2.
$$
\end{thm}
\begin{proof}
If $[\si]\in\Av_n([1234],[1243],[1342])$ then $[\si]$ avoids $[1243]$ and $[1342]$.  So, by the proof of Theorem~\ref{1243,1342}, we can write \begin{equation}
\label{sixy}
  \si=xy\rho  
\end{equation} 
where $\{x,y\}=\{1,n\}$ and $\rho$ is either increasing or decreasing.  Since $n\ge5$ we have $|\rho|\ge3$.  But $[\si]$ also avoides $[1234]$ and this forces $\rho$ to be decreasing.  So there are two choices for $[\si]$ depending on the values of $x$ and $y$.
\end{proof}

\begin{thm}
\label{1234,1243,1423}
We have
$$
\{[1234],[1243],[1423]\} \equiv \{[1234],[1342],[1423]\}
\equiv \{[1243],[1324],[1432]\} \equiv \{[1324],[1342],[1432]\}. 
$$
And for $n\ge2$
$$
\#\Av_n([1234],[1342],[1423])=n-1.
$$
\end{thm}
\begin{proof}
We will show that $T=T([1234],[1342],[1423])$ has production rules
\begin{align*}
(*)&\ra (1)(2),\\
(1)&\ra(1),\\
(2)&\ra(1)(2).
\end{align*}
Then, by induction, level $n\ge2$ of $T$ will contain one node of degree $2$ and $n-2$  nodes of degree $1$.  Checking the root is easy, so assume $n\ge3$.

By Theorem~\ref{1234,1243}, $T$ is a subtree of $T([1234],[1342])$.  So we just need to check which nodes of that tree also avoid $[1423]$.  As in the proof of that theorem, the site before $n$ in $[\si]$ at level $n$ in $T$ is still active since $4$ is not followed immediately by $3$ in $[1423]$.  Thus it suffices to show that both sites of $\de$ remain active, but only one in $\ep$ where $\de,\ep$ are defined by~\eqref{deep}.  Indeed, the two sites of $\de$ give rise to copies of $\de$ and $\ep$ at level $n+1$ of $T$.  But site $n-1$ of delta which was active in the larger tree is now inactive since inserting $n+1$ there gives the copy $[1,n+1,2,n]$ of $[1423]$.  This completes the proof.
\end{proof}

We now have, in comparison with the previous theorem, a nontrivial Wilf equivalence.
\begin{thm}
\label{1234,1324,1423}
We have
$$
\{[1234],[1324],[1423]\} \equiv \{[1324],[1423],[1432]\}. 
$$
And for $n\ge2$
$$
\#\Av_n([1234],[1324],[1423])=n-1.
$$
\end{thm}
\begin{proof}
It suffices to show that $T=T([1234],[1324],[1423])$ satisfies the same production rules as in the previous theorem.  Now $T$ is a subtree of $T([1234],[1324])$ which was constructed in the proof of Theorem~\ref{1234,1324}.  And we see in the usual way that the site before $n$ in any $[\si]$ remains active in $T$ because $4$ is not immediately followed by $3$ in $[1423]$.  

So it suffices to show, with $\de$ and $\ep$ as in~\eqref{deep2}, that  site $1$ remains active in $\de$, but not in $\ep$.  Indeed, inserting $n+1$ in this site of $\de$ just produces another descending sequence.  But in $\ep$ such a placement gives the copy $[1,n+1,n-1,n]$ of $[1423]$.
\end{proof}

We now have another nontrivial Wilf equivalence with Theorem~\ref{1234,1243,1324}.

\begin{thm}
\label{1243,1324,1342}
We have
$$
\{[1243],[1324],[1342]\} \equiv \{[1243],[1342],[1423]\}. 
$$
And for $n\ge4$
$$
\#\Av_n([1243],[1324],[1342])=3.
$$
\end{thm}
\begin{proof}
By Theorem~\ref{1243,1342}, we just need to show that exactly $3$ of the $4$ permutations $[\si]$ avoiding $\{[1243],[1342]\}$  also avoid $[1324]$.  These permutations are described in equation~\eqref{sixy}.  If $x=n$ and $y=1$ then $[\si]$ contains the copy $[n132]$ of this pattern.  It is also easy to check that the other three avoid it.
\end{proof}

We now have our last nontrivial Wilf equivalence for triples.
\begin{thm}
\label{1243,1324,1423}
We have
$$
\{[1243],[1324],[1423]\} \equiv \{[1324],[1342],[1423]\}. 
$$
And for $n\ge2$
$$
\#\Av_n([1324],[1342],[1423])=n-1.
$$
\end{thm}
\begin{proof}
Comparing the description of $\Av_n([1324],[1342])$ in
the proof of Theorem~\ref{1243,1324} and that of $\Av_n([1324],[1423])$ in the proof of Theorem~\ref{1324,1423}, we see that any 
$[\si]\in\Av_n([1324],[1342],[1423])$ can be put in the form
$$
\si=21[\de_k,\io_{n-k}]
$$
with $k\ge1$.  Also, $k=n-1$ and $k=n$ yield the same permutation.  So there are $n-1$ choices for $k$ and we are done.
\end{proof}

When $\#[\Pi]\ge4$ where $[\Pi]\sbs[\fS_4]$, the size of $\Av_n[\Pi]$ becomes constant for $n\ge5$.  And this size is trivial to calculate for $n\le 4$.  Furthermore, the description of the surviving permutations for large $n$ is easy to obtain given our previous proofs.  So we content ourselves with a listing of the equivalence classes and associated constants in Table~\ref{[Pi]big}.  Classes are separated by double horizontal line.  As usual, we do not consider classes containing both the increasing and decreasing permutations because of the cyclic Erd\H{o}s-Szekeres Theorem.

\begin{table}
    $$
    \begin{array}{l|c}
       [\Pi]  & \#\Av_n[\Pi]  \\
       \hline \hline
        \{[1234], [1243], [1324], [1342]\} & 1\\
        \{[1243], [1342], [1423], [1432]\} &\\
        \hline \hline
        \{[1234], [1243], [1324], [1423]\}& 2\\
        \{[1234], [1243], [1342], [1423]\}&\\
        \{[1234], [1324], [1342], [1423]\}&\\
        \{[1243], [1324], [1342], [1423]\}&\\
        \{[1243], [1324], [1342], [1432]\}&\\
        \{[1243], [1324], [1423], [1432]\}&\\
        \{[1324], [1342], [1423], [1432]\}&\\
        \hline\hline
        \{[1234], [1243], [1324], [1342], [1423]\}&1\\
        \{[1243], [1324], [1342], [1423], [1432]\}&\\
        \hline\hline
    \end{array}
    $$
    \caption{Wilf equivalence classes and cardinalities of  $\Av_n[\Pi]$ for certain $[\Pi]$ and $n\ge5$}
    \label{[Pi]big}
\end{table}

\section{Cyclic descent generating functions}
\label{cdg}

We will now consider the generating function for the number of cyclic descents over various avoidance classes $[\Pi]\sbs[\fS_4]$, starting with those defined by a single element.  We will sometimes use the characterizations given by Callan~\cite{cal:pac} for these classes to facilitate our work, and use the abbreviation
$$
D_n([\Pi])=D_n([\Pi];q) = \sum_{\si\in\Av_n[\Pi]} q^{\cdes\si}
$$
for the generating function.

To begin, we have  a lemma showing that trivial Wilf equivalences also give simple relationships between the corresponding generating functions.
\begin{lem}
\label{rcgf}
For any $[\Pi]$ We have
$$
D_n([\Pi]^r;q)=D_n([\Pi]^c;q)= q^n D_n([\pi];1/q)
$$
and
$$
D_n([\Pi]^{rc};q)=D_n([\Pi];q).
$$
\end{lem}
\begin{proof}
Reversing or complementing a permutation turns all cyclic descents into cyclic ascents and vice-versa.  Translating this into generating functions gives the first displayed equalities.  And the second displayed equation follows from the the previous display. 
\end{proof}

Now consider the possible $D_n([\pi])$ for $[\pi]\in[\fS_4]$.  We begin with the simplest case.

\begin{thm}
We have $D_n([1423];q)=q^n D_n([1324];1/q)$ where, for $n\ge2$,
$$
D_n([1324];q)= \sum_{k=1}^{n-1}\binom{n+k-3}{n-k-1} q^k.
$$
\end{thm}
\begin{proof}
We use Callan's characterization of this avoidance class to obtain a recursion for $D_n([1324])$.  If $[\si]\in\Av_n([1324])$ and $n\ge 3$ then write $\si=\si_1\si_2\ldots\si_{n-1}n$.  Let $k$ be the index such that $\si_k=n-1$.  There are two cases.

If $k=n-1$ then $\si=\tau,n-1,n$ where $[\tau,n-1]\in\Av_{n-1}([1324])$ and this is a bijection.  Since $\cdes[\si]=\cdes[\tau,n-1]$, this case contributes $D_{n-1}([1324])$ to the recursion.

If $1\le k\le n-2$ then this forces
$$
\si=2314[\io_{k-1},1,\tau,1]
$$
for some $\tau$ such that  $[\tau n]$ avoids $[1324]$.  Because of the extra descent caused by $n-1$ we have $\cdes[\si]=1+\cdes[\tau n]$.  So this case gives a contribution of $\sum_{k=1}^{n-2} q D_{n-k}([1324])$.

Putting everything together, we have
$$
D_n([1324])= D_{n-1}([1324]) + \sum_{k=1}^{n-2} q D_{n-k}([1324]).
$$
for $n\ge3$ and $D_2([1324])=q$.   It is now a simple manner of manipulating binomial coefficients to show that the formula given in the theorem satisfies this initial value problem.
\end{proof}

For the next case, we will use a characterization of the class different from the one found by Callan.  This will permit us to avoid the use of a recurrence.
\begin{lem}
\label{1342}
Suppose $[\si]\in[\fS_n]$ and write $\si=1\rho n\tau$.  We
have $[\si]\in\Av_n([1342])$ if and only if the following three conditions are satsified:
\begin{enumerate}
    \item[(a)] $\rho$ and $\tau$ both avoid $\{213,231\}$,
    \item[(b)] $\max\rho<\min\tau$,
    \item[(c)] there is not both a descent in $\rho$ and an ascent in $\tau$.
\end{enumerate}
\end{lem}
\begin{proof}
For the forward direction, suppose $[\si]\in\Av_n([1342])$.  Condition (a) is true since if either $\rho$ or $\tau$ contains $213$ then, together with $n$, we have that $[\si]$ contains $[2134]$.  Similarly, if either contains $231$ then $[\si]$ contains the forbidden pattern by prepending the $1$.  As far as (b), if there is $y>x$ with $y\in\rho$ and $x\in\tau$ then $[1ynx]$ is a copy of $[1342]$.  Finally for (c), if there were a descent in $\rho$ and an ascent in $\tau$ then, because of (b), putting them together would again give a copy of the pattern to avoid.

The converse is similar where one assumes that a copy of $[1342]$ exists and then considers all the different intersections it could have with $1$, $\rho$, $n$, and $\tau$.  We leave the details to the reader.
\end{proof}

In order to use this lemma, we will need a result about the ordinary descent statistic on linear permutations avoiding $\{123,231\}$.  The  next result is a specialization of Proposition 5.2 of the paper of Dokos, Dwyer, Johnson, Sagan, and Selsor~\cite{DDJSS:pps} and so the proof is ommited. 
\begin{lem}[\cite{DDJSS:pps}]
\label{DDJSS}
We have

\vs{5pt}

\eqed{
\sum_{\si\in\Av_n(213,231)} q^{\des\si} = (1+q)^{n-1}.
}
\end{lem}

We need one last well-known definition.  Call a polynomial $f(q)=\sum_{k=0}^n a_k q^k$ of degree $n$ {\em symmetric} if $a_k=a_{n-k}$ for all $0\le k\le n$.  Note that $f(q)$ of degree $n$ is symmetric if and only if
\begin{equation}
\label{sym}
  q^n f(1/q) = f(q).  
\end{equation}
\begin{thm}
We have $D_n([1243];q)=D_n([1342];q)$ where, for $n\ge2$,
$$
D_n([1342];q) = 2q(1+q)^{n-2}- q\cdot\frac{1-q^{n-1}}{1-q}
$$
is symmetric.
\end{thm}
\begin{proof}
It is easy to prove from the explicit form of $D_n([1342])$ that it satisfies equation~\eqref{sym} and so is symmetric.  So once this is proved, the equality of the two generating functions follows from Lemma~\ref{rcgf}.

We adopt the notation of Lemma~\ref{1342} and let $\si_k=n$ where $2\le k\le n$.  We will consider cases depending on whether $\rho$ or $\tau$ is empty.  If $\rho=\emp$ then by Lemma~\ref{1342} (a) and Lemma~\ref{DDJSS} we have that the generating function for the possible linear $\tau$ is $(1+q)^{n-3}$.  Also, $\cdes[\si]=2+\des\tau$ by the form of $\si$, so the contribution of such $[\si]$ to $D_n([1342])$ is $q^2(1+q)^{n-3}$.  In an analogous way, we see that those $[\si]$ with $\tau=\emp$ yield $q(1+q)^{n-3}$.  Adding these, we have a total of $q(1+q)^{n-2}$ so far.

We now assume that $\rho,\tau$ are both nonempty so that 
$3\le k\le n-1$.  By parts (b) and (c) of Lemma~\ref{1342}, either $\rho$ must be an increasing subsequence of consecutive integers or $\tau$ must be a decreasing one.  Using Lemma~\ref{DDJSS} again, we see that in the first subcase a contribution of $q^2(1+q)^{n-k-1}$ is obtained.  And in the second, taking into account the descents in $\rho$, the contribution is $q^{n-k+1}(1+q)^{k-3}$.  However, these two subcases overlap when $\rho$ is increasing and $\tau$ is decreasing.  So we must subtract $q^{n-k+1}$.

Thus we get a grand total of
$$
D_n([1342]) = q(1+q)^{n-2} + \sum_{k=3}^{n-1}[q^2(1+q)^{n-k-1}+q^{n-k+1}(1+q)^{k-3}-q^{n-k+1}].
$$
Summing the geometric series and simplifying completes the proof.
\end{proof}

For the avoidance class of the increasing (or decreasing) pattern in $[\fS_4]$, we will need another concept.  Given sequences $\rho$ and $\tau$ of distinct integers, their {\em shuffle set} is
$$
\rho\shu\tau =\{\si\ :\ 
\text{$|\si|=|\rho|+|\tau|$ and both $\rho,\tau$ are subsequences of $\si$}\}.
$$
For example, 
$$
12\shu 34 =\{1234, 1324, 1342, 3124, 3142, 3412\}.
$$
In the statement of the next result we make the usual convention that $\binom{n}{k}=0$ if $k>n$.
\begin{thm}
We have $D_n([1234];q)= q^n D_n([1432];1/q)$ where, for $n\ge2$,
$$
D_n([1432];q) = q+(2^{n-1}-n) q^2 +\sum_{j\ge3}\binom{n}{2j-1}q^j.
$$
\end{thm}
\begin{proof}
We use Callan's description of the avoidance for $[1234]$ translated by complementation to apply to $[1432]$.  We are going to derive a recursion for $D_n([1432];q)$.  If $[\si]\in\fS_n[1432]$  then suppose $\si_n=1$ and $\si_k=2$ for some $1\le k\le n-1$.  There are three cases.

If $k=1$ then there is a bijection between such $[\si]$ and $\Av_{n-1}[1432]$ obtained by removing $1$ and taking the order isomorphic cyclic permutation on $[n-1]$.  Since $2$ immediately follows $1$ cyclically in $[\si]$, the descent into $1$ remains a descent after applying the map.  So the contribution of this case is $D_{n-1}([1432];q)$.

Now suppose that $2\le k\le n-1$ and write
$$
\si=\rho 2\tau 1.
$$
where $|\rho|=k-1$, $|\tau|=n-k-1$.  As Callan proves, $\rho$ must be increasing.  So there are two more cases depending upon whether the elements of $\rho$ are consecutive or not.  Suppose first that they are not consecutive.  In this case, $\tau$ must also be increasing so $\cdes[\si]=2$.  To compute the number of such $\si$, note that once the elements of $\rho$ have been picked from $[3,n]$, all of $\si$ is determined. The total number of nonempty subsets of this interval is $2^{n-2}-1$.  And those which consist of consecutive integers are determined by their minimum and maximum element, which could be equal.  So there are $\binom{n-1}{2}$ subsets to exclude.  The contribuion of this case is then
$$
\left(2^{n-2}-\binom{n-1}{2}-1\right)q^2.
$$

Finally we consider the case when $\rho\neq\emp$ is consecutive (and still increasing), say with minimum $m+1$ and maximum $M-1$.  
Note that if $l=|\tau|$ then $0\le l\le n-3$.
Callan shows that the possible $\tau$ are the elements of 
$(34\ldots m)\shu(M,M+1,\ldots,n)$.  Since a permutation can be written as a shuffle in many ways, the same shuffle could occur for different $\rho$.  So it will be convenient to color the elements of the second sequence by marking them with a hat.  Thus the $\si$ in this case are in bijection with colored shuffles 
$(34\ldots m)\shu(\widehat{M},\widehat{M+1},\ldots,\widehat{n})$.  It will also be convenient to consider these as corresponding to the sequences $2\tau$ by prepending a $2$ to each shuffle and considering $2$ as an uncolored element.
Set $S$ be the set of such sequences $s=2s_2 s_3\ldots s_{l+1}$ 
where $l,m,M$ are allowed to vary over all possible values.  Note that if $s$ corresponds to $\si$ then $\des\si=2+\des s$.  To compute $\des s$, we consider the transition indices
$$
\Tr s = \{i \mid \text{$s_i$ is colored and $s_{i+1}$ is not, or vice-versa}\}.
$$
For example, if $s=23\widehat{6}45\widehat{7}\widehat{8}$ then $\Tr s=\{2,3,5\}$.  It is easy to see that the map $\Tr:S\ra 2^{[l]}$, the range being all subsets of $[l]$, is a bijection.  Also, every other transition index of $s$ starting with the second corresponds to a descent.  So, using the round down function, $\des s = \fl{\#\Tr s/2}$.  We can now complete this case  using $i=\#\Tr s$ to see that the contribution  is
\begin{align*}
\sum_{l=0}^{n-3}\sum_{i=0}^l \binom{l}{i} q^{\fl{i/2}+2}
&=\sum_{i=0}^{n-3} q^{\fl{i/2}+2} \sum_{l=i}^{n-3} \binom{l}{i}\\
&=\sum_{i=0}^{n-3} \binom{n-2}{i+1} q^{\fl{i/2}+2}\\
&=q^2 \sum_{j\ge0}\left[\binom{n-2}{2j+1}+\binom{n-2}{2j+2}\right]q^j\\
&=q^2 \sum_{j\ge0}\binom{n-1}{2j+2}q^j.
\end{align*}

Putting all the cases together we have
$$
D_n([1432];q) = D_{n-1}([1432];q) + q^2 
\left[2^{n-2}-\binom{n-1}{2}-1 +\sum_{j\ge0}\binom{n-1}{2j+2}q^j\right].
$$
As usual, the routine verification that our desired formula satisfies this recursion and the initial condition is left to the reader.
\end{proof}

We now turn to the cyclic descent polynomials for pairs in $[\fS_4]$.  To simplify notation, for any polynomial $f(q)$ and  $n\in\bbN$ we let 
$$
f^{(n)}(q)=q^n f(1/q).
$$
\begin{thm}
\label{DnDob}
We have the following descent polynomials.
\begin{enumerate}
\item[(a)]  We have
$$
D_n([1234],[1243]) = D_n([1342],[1432]) = 
D_n^{(n)}([1243],[1432]) = D_n^{(n)}([1234],[1342]).
$$
And for $n\ge3$
$$
D_n([1234],[1342];q) = (2n-5)q^{n-2} + q^{n-1}.
$$
\item[(b)] We have
$$
D_n([1423],[1432])=D_n^{(n)}([1234],[1324]).
$$
And for $n\ge3$
$$
D_n([1234],[1324];q) = (2n-5)q^{n-2} + q^{n-1}.
$$
\item[(c)] We have
$$
D_n([1324],[1432]) = D_n^{(n)}([1234],[1423]).
$$
And for $n\ge1$
$$
D_n([1234],[1423];q)=q^{n-1}+\binom{n-1}{2}q^{n-2}.
$$
\item[(d)]  We have
$$
D_n([1243],[1423])=D_n([1342],[1423])
=D_n^{(n)}([1243],[1324]) = D_n^{(n)}([1324],[1342]).
$$
And for $n\ge1$
$$
D_n([1324],[1342];q)=q+\sum_{k=2}^{n-1} (n-k) q^k.
$$
\item[(e)] For $n\ge4$ we have
$$
D_n([1243],[1342];q)=q+q^2+q^{n-1}+q^{n-2}.
$$
\item[(f)] For $n\ge3$ we have
$$
D_n([1324],[1423];q)=q(1+q)^{n-2}.
$$
\end{enumerate}
\end{thm}
\begin{proof}
We will only prove (a) as the others follow easily in a similar fashion from the descriptions of the avoidance classes in Section~\ref{pad}.  We adopt the notation of the proof of Theorem~\ref{1234,1243}.  

We will use the description of the generating tree to obtain a recursion for $D_{n+1}[1243],[1432])$.  Note that
if $n+1$ is inserted in site $i$ of $\si$ to form $\si'$ then
$$
\cdes[\si']=\case{\cdes[\si]}{if $i$ is  a cyclic descent,}
{\cdes[\si]+1}{if $i$ is a cyclic ascent.}
$$
Since the site before $n$ is always active, these children will give a contribution of $qD_n([1243],[1432])$ because such a site is a cyclic ascent.  In $\de$ and $\ep$, insertion in the other active site gives permutations with $n-1$ descents.  So
$$
D_{n+1}[1243],[1432]) = 2 q^{n-1}+qD_n([1243],[1432]).
$$
It is now easy to check that the formula in (a) satisfies this recursion and is also valid at $n=3$, completing the proof.
\end{proof}

For classes avoiding $3$ or more patterns, we will only write down the results for those which are not eventually constant.  The interested reader can easily compute the polynomials for the remaining classes.  We also content ourselves with stating the polynomial for one member of every trivial Wilf equivalence class since the rest can be computed from Lemma~\ref{rcgf}.
\begin{thm}
We have the  descent polynomials
$$
D_n([1234],[1342],[1423];q) 
=D_n([1234],[1324],[1423];q)
=(n-2) q^{n-2} + q^{n-1}
$$
and
$$
D_n([1324],[1342],[1423];q) = q\cdot\frac{1-q^{n-1}}{1-q}
$$
for $n\ge2$.\hqed
\end{thm}

\section{Open problems and concluding remarks}
\label{opc}

We collect here various areas for future research in the hopes that the reader will be interested in pursuing this work.

\subsection{Longer patterns}

There has been very little work about containment and avoidance  for cyclic patterns of length longer than $4$.  Of course, the cyclic Erd\H{o}s-Szekeres Theorem, Theorem~\ref{EScyc} above, is one such result.  There is also a paper of Gray, Lanning and Wang~\cite{GLW:pcc1} where the authors consider cyclic packing (maximizing the number of copies of a given pattern among all the permutations $[\si]\in[\fS_n]$ for some $n$) and superpatterns (permutations containing all the patterns $[\pi]\in[\fS_k]$ for some $k$).
It would be interesting to see if there are nice enumerative formulas for  classes consisting of cyclic patterns of length $5$ and up.

\subsection{Other statistics}

One could study other cyclic statistics.  For example,
the {\em peak set} of a linear permutation is
$$
\Pk\pi = \{i \mid \pi_{i-1}<\pi_i>\pi_{i+1}\}
$$
with corresponding {\em peak number}
$$
\pk\pi=\#\Pk\pi.
$$
Peaks are an important part of Stembridge's theory of enriched $P$-partitions~\cite{ste:epp} where $P$ is a partially ordered set.  On the enumerative side, the study of permutations which have a given peak set has been a subject of current 
interest~\cite{BBPS:mew,BBS:pgp,BFT:crp,CDOPZ:nps,DHIO:ppp1,DHIO:ppp2,DHIP:psc}.  Now define the {\em cyclic peak number} to be
$$
\cpk[\pi]= \#\{i \mid \text{$\pi_{i-1}<\pi_i>\pi_{i+1}$ where subscripts are taken modulo $n$}\}.
$$
As with $\cdes$, this is well defined since it is independent of the choice of representative of $[\pi]$.  There should be interesting generating functions for the distribution of $\cpk$ over avoidance classes, or even for the joint distribution of $\cdes$ and $\cpk$.  As evidence, we prove one such result.
\begin{thm}
For $n\ge3$
$$
\sum_{[\si]\in\Av_n([1234],[1342])}q^{\cdes[\si]}t^{\cpk[\si]}
= q^{n-2}t+(2n-6)q^{n-2}t^2 + q^{n-1}t.
$$
\end{thm}
\begin{proof}
Let $F_n(q,t)$ denote the desired generating function.  We proceed as in the proof of Theorem~\ref{DnDob} (a) to find a recursion for $F_{n+1}(q,t)$.  Since the largest element of $[\si]$ is always a cyclic peak, inserting $n+1$ before $n$ does not change $\cpk$.  So this contributes $q F_n(q,t)$ to the recursion.  For $\de$ and $\ep$, inserting $n+1$ in the other active site increases the number of peaks to $2$.  So the contribution from these cases is $2 q^{n-1} t^2$.  In summary
$$
F_{n+1}(q,t) = 2 q^{n-1} t^2 + q F_n(q,t)
$$
and the desired polynomial is easily seen to be the solution.
\end{proof}

In a recent paper Adin, Gessel, Reiner, and Roichman~\cite{AGRR:cqf} defined a cyclic analogue of the Hopf algebra of quasisymmetric functions.  In this context the cyclic descent set of a linear permutation arises naturally in the description of the product in this algebra.  They also raise the following intriguing question.
\begin{question}
Find an analogue of the major index for cyclic permutations that has nice properties, such as a generating function over $[\fS_n]$ which factors nicely as does the generating function for the ordinary major index over $\fS_n$. 
\end{question}

\subsection{Vincular patterns}

The study of vincular patterns was originated by Babson and Steingr\'{\i}msson~\cite{BS:gpp} and has since become a mainstay of the pattern field.  We consider $\pi$ as a {\em vincular pattern} if one only counts occurrences in $\si$ where certain adjacent elements of $\pi$ must also be adjacent in the copy in $\si$.  Such adjacent elements are overlined in $\pi$.  For example, $\si=24513$
contains two copies of $\pi=132$, namely $243$ and $253$.  But only $243$ is a copy of $\ol{13}2$.  Avoidance and Wilf equivalence are defined in the obvious way.  These notions and the corresponding notation carry over to cyclic patterns without change.  There are undoubtedly nice results which can be proven about vincular cyclic patterns.  As an example, we show how one vincular class is enumerated by the Catalan numbers.
\begin{thm}
We have
$$
[13\ol{24}]\equiv [1\ol{42}3]\equiv[\ol{13}24]\equiv[2\ol{31}4]. 
$$
And for $n\ge1$
$$
\#\Av_n[13\ol{24}]=C_{n-1}.
$$
\end{thm}
\begin{proof}
The Wilf equivalences are trivial. To prove the Catalan formula, suppose that $[\si]\in\Av_n[13\ol{24}]$ for $n\ge2$ and write
$\si$ so that $\si_n=n$ and $\si_{n-1}=m$ for some $m\in[n-1]$.  First notice that $\si=\rho\tau m n$ where $\rho$ and $\tau$ are permutations of $[m+1,n-1]$ and $[m-1]$, respectively.  For if there are $x<m<y<n$ with $x$ before $y$ in $\si$ then $[xymn]$ is a copy of $[13\ol{24}]$. Furthermore, it is clear that $[m\rho]$ and $[\tau m]$ must avoid the forbidden pattern.

We claim the if $\si=\rho\tau m n$ where $\rho$ and $\tau$ obey the restrictions of the previous paragraph then $[\si]$ avoids $[13\ol{24}]$.  Suppose, towards a contradiciton, that a  copy $[\ka]=[wyxz]$ 
exists with $wyxz$ order isomorphic to $13\ol{24}$.
Consider the elements $x$ and $z$ which play the roles of $2$ and $4$.  The possibility that they are $m$ and $n$, respectively, is ruled out by the fact that every element of $\rho$ is larger than every element of $\tau$.  If  $z\in\tau m$ then all of $\ka$ must be in this subsequence since $z$ is the largest element of the copy.  But this is impossible since $[\tau m]$ avoids the bad pattern.  Finally, suppose  $z\in\rho$. This forces  $x\in\rho$ since it is comes cyclically just before $z$, and $n$ is too large to be $x$.  We must also have $y\in\rho$ since $x<y<z$.  But now there is no possible choice for $w$.  Indeed, if $w\in [m\rho]$ then  $[\ka]$ is in this subsequence, contradicting our assumption.  And if $w\in\tau$ then it could be replaced by $m$ since $x,y,z>m$, yielding the same contradiction as before.

From the first two paragraphs we immediately get the recursion
$$
\#\Av_n[13\ol{24}]=\sum_{m=1}^{n-1} 
\#\Av_m[13\ol{24}]\cdot \#\Av_{n-m}[13\ol{24}].
$$
From this the Catalan enumeration follows by induction.
\end{proof}

It appears that sometimes rather than trying to find the size of the avoidance class directly, it may be easier to use exponential generating functions.  Given a set of (possibly vincular) patterns $[\Pi]$, let
$$
E[\Pi]=\sum_{n\ge 0} \#\Av_n[\Pi] \frac{x^n}{n!}.
$$
We have the following conjectures for two vincular avoidance classes.  Once the corresponding differential equation is proved, an explicit solution can easily found using separation of variables.
\begin{conj}
We have the following.
\begin{enumerate}
    \item If $E=E[\ol{123}]$ then
    $$
    E' = E^2 - E + 1.
    $$
    \item If $E=E[\ol{213}]$ then
    $$
    E' = e^{E-\frac{x^2}{2}}.
    $$
\end{enumerate}
\end{conj}



\bibliographystyle{alpha}

\newcommand{\etalchar}[1]{$^{#1}$}

\end{document}